\documentclass{article}
\usepackage{amsmath}
\usepackage{amsthm}
\usepackage{amssymb}
\usepackage[ansinew]{inputenc}
\usepackage{url}
\usepackage[all,dvips,arc,curve,color,frame]{xy}
\input xy
\xyoption{all}

\newtheorem{lemma}{Lemma}[section]

\newtheorem{theorem}[lemma]{Theorem}
\newtheorem{cor}[lemma]{Corollary}
\newtheorem{prob}[lemma]{Problem}
\newtheorem{prop}[lemma]{Proposition}
\newtheorem{defi}[lemma]{Definition}

\DeclareMathOperator\ima{im}

\DeclareMathOperator\id{id}

\newcommand{\con}{\mathrm{Con}}

\newcommand{\miff}{\Leftrightarrow}
\newcommand{\lan}{\langle}
\newcommand{\ran}{\rangle}
\newcommand{\al}{\alpha}

\newcommand{\lam}{\lambda}

\newcommand{\vep}{\varepsilon}

\newcommand{\si}{\sigma}

\newcommand{\ty}{\mbox{\tiny $Y$}}
\newcommand{\tz}{\mbox{\tiny $Z$}}
\newcommand{\tyz}{\mbox{\tiny $Y\!\cap\! Z$}}
\newcommand{\ey}{\vep_{\ty}}
\newcommand{\ez}{\vep_{\tz}}
\newcommand{\eyz}{\vep_{\tyz}}

\newcommand{\End}{\mathop{\mathrm{End}}\nolimits}
\newcommand\A{\Bbb A}
\newcommand\lia{\A^{\!l}}
\newcommand\aia{\A^{\!a}}
\newcommand\mia{\A^{\!m}}
\newcommand\qia{\A^{\!q}}

\newcommand\eia{\A^{\!e}}
\newcommand\gia{\A^{\!g}}
\newcommand\clf{F_{\!\mathrm{cl}}}
\newcommand\clfa{(F_1)_{\mathrm{cl}}}
\newcommand\clfb{(F_2)_{\mathrm{cl}}}

\title{The Largest Subsemilattices of the Endomorphism Monoid of an Independence Algebra}

\author{Jo\~ao Ara\'ujo\\
  {\small Universidade Aberta, R. Escola Polit\'{e}cnica, 147}\\
  {\small 1269-001 Lisboa, Portugal}\\{\footnotesize \&}\\
  {\small Centro de \'{A}lgebra, Universidade de Lisboa}\\
  {\small 1649-003 Lisboa, Portugal, jaraujo@ptmat.fc.ul.pt}
  \and Wolfram Bentz\\
  {\small Centro de \'{A}lgebra, Universidade de Lisboa}\\
  {\small 1649-003 Lisboa, Portugal, wfbentz@fc.ul.pt}
\and Janusz Konieczny\\
{\small Department of Mathematics, University of Mary Washington}\\
{\small Fredericksburg, Virginia 22401, USA, jkoniecz@umw.edu}}
\date{}
\begin{document}
\maketitle

\begin{abstract}
An algebra $\A$ is said to be an independence algebra
if it is a matroid algebra and every map $\al:X\to A$, defined on a  basis $X$ of $\A$,
can be extended to an endomorphism of~$\A$.
These algebras are particularly  well behaved generalizations of vector spaces,
and hence they naturally appear in several branches of mathematics such as model theory, group theory, and semigroup theory.

It is well
known that matroid algebras have a well defined notion of dimension.
Let~$\A$ be any independence algebra of finite dimension $n$, with at least two elements.
Denote by $\End(\A)$ the monoid of endomorphisms of $\A$.
We prove that a largest subsemilattice of $\End(\A)$ has either $2^{n-1}$ elements
(if the clone of $\A$ does not contain any constant operations) or $2^n$ elements (if
the clone of $\A$ contains constant operations).
As corollaries, we obtain formulas for the size of the largest subsemilattices of:
some variants of the monoid of linear operators of a finite-dimensional vector space, the monoid
of full transformations on a finite set $X$, the monoid of partial transformations on $X$,
the monoid of endomorphisms of a free $G$-set with a finite set of free generators, among others.

The paper ends with a relatively large number of problems that might attract attention of experts in linear algebra, ring theory, extremal combinatorics, group theory, semigroup theory,
universal algebraic geometry, and universal algebra.
\vskip 2mm

\noindent\emph{$2010$ Mathematics Subject Classification\/}. 08A35, 20M20.

\vskip 2mm


\end{abstract}

\section{Introduction}
\setcounter{equation}{0}

Let $\A=\lan A;F\ran$  be an algebra (as understood in universal algebra \cite{Gr68}).
We say that $\A$ is a \emph{matroid algebra} if  the closure operator \emph{subalgebra generated by},
denoted $\langle \cdot \rangle$, satisfies the \emph{exchange property}, that is, for all $X\subseteq A$ and $x,y\in A$,
\begin{equation}
\label{eq1}
x \in \langle X\cup \{y\}\rangle \mbox{ and } x\not\in X\Rightarrow y \in \langle X\cup \{x\}\rangle.
\end{equation}
(In certain contexts of model theory, a closure
system satisfying the exchange property is called a \emph{pre-geometry}.)
A set $X\subseteq A$ is said to be \emph{independent} if $X$ has no redundant elements, that is,
$X$ is a minimal generating set for the subalgebra it generates: for all $x\in X$, $x\not\in \langle X\setminus \{x\}\rangle.$

By standard arguments in matroid theory, we know that a matroid algebra has a \emph{basis} (an independent generating set),
and all bases have the same cardinality; thus matroid algebras admit a notion of dimension,
defined as the cardinality of one (and hence all) of its bases.
An \emph{independence algebra} is a matroid algebra satisfying the \emph{extension property}, that is,  every map $\al:X\to A$,
defined on a basis $X$ for $\A$, can be extended to an endomorphism of $\A$.
Examples of independence algebras are vector spaces, affine spaces (as defined below), unstructured sets, and free $G$-sets.

The class of independence algebras was introduced by Gould in 1995 \cite{gould}. Her motivation
was to understand the properties shared by vector spaces and sets that result in similarities
in the structure of their monoids of endomorphisms.
As pointed out by Gould, this notion goes back to the 1960s, when the class of
$v^*$-algebras was defined by Narkiewicz \cite{Na61}.
(The ``$v$'' in $v^*$-algebras stands for ``vector''
since the $v^*$-algebras were primarily seen as generalizations of the vector spaces.) In fact,
the $v^*$-algebras can be defined as the matroid algebras with the extension property \cite{Na64}, just like independence algebras, but with a  slight difference. In the context of independence algebras, the subalgebra generated by the empty set is
the subalgebra generated by all nullary operations; while in $v^*$-algebras, it is the subalgebra generated by the images of all constant operations. The effect of this difference, so tiny that it has gone unnoticed by previous authors,
is that there do exist $v^*$-algebras $\A=\lan A;F\ran$ that are not independence algebras, namely exactly those
for which:
\begin{enumerate}
\item $|A| \ge 2$,
\item $\A$ does not have any nullary operations,
\item every element of $A$ is the image of some constant operation from the clone of $\A$.
\end{enumerate}

We remark that our result can be extended, with little effort, to include all $v^*$-algebras,
and hence it also holds for this slightly larger class. The key observation is that if a $v^*$-algebra is not an independence algebra,
then its monoid of endomorphisms is trivial.

By the end of the 1970s, G{\l}azek wrote a survey paper on these and related algebras,
including a bibliography of more than 800 items \cite{glazek}. (See also \cite{JMS,ArWe} and the references therein.) About ten years later,
independence algebras naturally appeared in semigroup theory. (For a survey, see \cite{arfo}; see also \cite{Ar2,Ar4,Ar1,Ar5,Ar3,AS,AS2,cameronSz,F1,F2} for some results on independence algebras and semigroups.)

Between the 1960s
(when $v^*$-algebras were introduced by experts in universal algebra) and the 1990s
(when they were rediscovered by experts in semigroup theory),
these algebras
played a very important role in model theory. Givant in the U.S. \cite{gi1,gi2,gi2.1,gi3,gi4,gi5,gi6}
and Palyutin in Russia \cite{pa}, independently solved an important classification problem in model theory,
and their solution involved independence algebras.
(For a detailed account of the importance of independence algebras for model theory, see \cite{ArEdGi}.)

Independence algebras have a structure rich enough to allow classification theorems.
One, due to Cameron and Szab\'o \cite{cameronSz}, provides a classification of finite independence algebras.
Another, due to Urbanik \cite{Ur63,Ur65,urbanik},
classifies all $v^*$-algebras that have no nullary operations.
Since our goal is to prove a theorem about all finite-dimensional
independence algebras, we use Urbanik's classification and, in a
separate section, handle the case of independence algebras with nullary operations (noting again that
all independence algebras are $v^*$-algebras).

A \emph{semilattice} is a commutative semigroup consisting entirely of idempotents. That is, a semigroup
$S$ is a semilattice if and only if for all $a,b\in S$, $aa=a$ and $ab=ba$. A semilattice
can also be defined as a partially ordered set $(S,\leq)$ such that the greatest lower
bound $a\wedge b$ exists for all $a,b\in S$. Indeed, if $S$ is a semilattice, then $(S,\leq)$, where
$\leq$ is a relation on $S$ defined by $a\leq b$ if $a=ab$, is a poset with $a\wedge b=ab$
for all $a,b\in S$. Conversely, if $(S,\leq)$ is a poset such that $a\wedge b$ exists for all $a,b\in S$,
then $S$ with multiplication $ab=a\wedge b$ is a semilattice \cite[Proposition~1.3.2]{Ho95}). This paper's study of commuting idempotent endomorphisms is closely linked to the study of centralizers of idempotents \cite{andre,ArKo03,ArKo04,Ko02} and
general centralizers of transformations \cite{ArKo13,KoLi98}.

For an independence algebra $\A$, denote by $\End(\A)$ the monoid of endomorphisms of $\A$. The aim
of this paper is to prove the following theorem.

\begin{theorem}
\label{imain}
Let $\A=\lan A;F\ran$ be an independence algebra of finite dimension $n$, with $|A|\geq2$. Then the largest subsemilattices
of $\End(\A)$ have either $2^{n-1}$ or $2^n$ elements, with the latter happening exactly
when  the clone of $\A$ contains constant operations.
\end{theorem}

Since vector spaces and unstructured sets are independence algebras, we have
the following corollaries.

 \begin{cor}
\label{cor1}
 If $V$ is a vector space of finite dimension $n$, then the largest subsemilattices of $\End(V)$ have $2^n$ elements.
 \end{cor}

 \begin{cor}
\label{cor2}
 If $X$ is a nonempty finite set of size $n$, then the largest subsemilattices of $T(X)$, the
 monoid of full transformations on $X$, have $2^{n-1}$ elements.
 \end{cor}

The monoid $P(X)$ of all partial transformations on a finite set $X=\{1,\ldots,n\}$
is isomorphic to the endomorphism monoid of
the independence algebra $\A=\lan X\cup\{0\};\{f\}\ran$, in which $f(x)=0$ is a constant operation
in the clone of $\A$.
Therefore, we have another corollary.

\begin{cor}
\label{cor3}
 If $X$ is a nonempty finite set of size $n$, then the largest subsemilattices of $P(X)$ have $2^{n}$ elements.
\end{cor}

In recent years,
many papers have been devoted to connections between a given
algebra and some graphs induced by the algebra. Examples include
the zero-divisor graph (more than 150 papers have been written on these graphs in the last 10 years),
the commuting graph (more than $50$ papers in the last 10 years), and the power graph.
The goal is to investigate to which extent the induced graph shapes the structure of the algebra itself.
Since the endomorphism monoid of an independence algebra is deeply connected with its idempotents
(for a finite-dimensional independence algebra, the singular endomorphisms are  idempotent generated \cite{Ar1,F1}; see also \cite{andre2,ArCa1,ArCa2,ArMC}),
it is natural to consider the idempotent commuting graph of $\End(\A)$.
Thus, from this point of view, Theorem~\ref{imain} provides the clique number of such a graph.

Another interesting connection is provided by \cite{Ed}, where pairs of commuting idempotent endomorphisms of a group are used to describe all associative interchange rings and nearrings.

The paper is organized as follows. In Section~\ref{spre}, we provide relevant definitions and terminology.
Since we have not been able to find a general argument that works for all independence algebras,
we rely on the classification of independence algebras without nullary operations  obtained by Urbanik.
We present this classification in Section~\ref{scla}.
In Sections~\ref{scia}--\ref{saia}, we prove our theorem for each type of independence algebras without nullary operations.
In Section~\ref{smth}, we use these results and a proposition that links independence algebras with and without nullary operations
to prove our theorem in all generality. Finally, in Section~\ref{spro}, we present some problems.

\section{Preliminaries}
\label{spre}
In this section, we provide terminology and definitions on independence
algebras that we will need in the paper.

Let $\A=\lan A;F\ran$  be an algebra, that is, $A$ is a nonempty set (called the \emph{universe}) and $F$ is a set
of operations on $A$ (called the \emph{fundamental operations}) \cite[page~8]{Gr68}.
As customary, we will identify a nullary operation $f()=a$ with $a\in A$.
A function $\al:A\to A$ is called an \emph{endomorphism} of $\A$ if it preserves
all operations in $F$, that is, for all $k\geq0$, if $f$ is a $k$-ary
fundamental operation and $a_1,\ldots,a_k\in A$, then $\al(f(a_1,\ldots,a_k))=f(\al(a_1),\ldots,\al(a_k))$.
The set $\End(\A)$ is the monoid under the composition of functions.

For every $k\geq1$ and $1\leq i\leq k$, we will
denote by $p^k_i$ the $k$-ary projection on the $i$th coordinate, that is, $p^k_i(x_1,\ldots,x_k)=x_i$.
The \emph{clone} of $\A$ is the smallest set of operations on $A$ that contains $F$ and all
projection operations, and is closed under generalized composition \cite[page~45]{Gr68}.
We will denote the clone of $\A=\lan A;F\ran$ by $\clf$. Note that our definition of clone includes nullary functions,
contrary to some authors.

\begin{defi}
\label{deq}
{\rm
We say that algebras $\A_1=\lan A;F_1\ran$ and $\A_2=\lan A;F_2\ran$ are \emph{equivalent}
if $\clfa=\clfb$ \cite[page~45]{Gr68}.
}
\end{defi}

For a nonempty subset $X$ of $A$, we denote by $\lan X\ran$ the subalgebra of $\A$ generated by $X$ \cite[page~35]{Gr68}.
Let $\con$ be the set of nullary operations in $F$. As in \cite[page~35]{Gr68}, we extend the definition of $\lan X\ran$
to the empty set: $\lan\emptyset\ran=\lan\con\ran$ if $\con\ne\emptyset$, and $\lan\emptyset\ran=\emptyset$ if $\con=\emptyset$.

A set $X\subseteq A$ (possibly empty) is said to be \emph{independent}
if for all $x\in X$, $x\not\in \langle X\setminus \{x\}\rangle$.
The exchange property (\ref{eq1}) has several equivalent formulations in terms of independent sets \cite[page~50]{mcmct}.

 \begin{prop}
\label{506}
 For any  algebra $\A$, the following conditions are equivalent:
 \begin{itemize}
  \item [\rm(a)] $\A$ satisfies the exchange property {\rm(\ref{eq1})};
  \item [\rm(b)] for all $X\subseteq A$ and $a\in A$, if $X$ is
 independent and $a\notin\langle X\rangle$, then $X\cup \{a\}$ is independent;
  \item [\rm(c)] for all $X,Y\subseteq A$, if $Y$ is a maximal independent subset
 of $X$, then $\langle X\rangle=\langle Y \rangle$;
  \item [\rm(d)] for all $X,Y\subseteq A$, if $Y$ is an independent subset of $X$,
 then there is an independent set $Z$ with $Y\subseteq Z\subseteq X$ and
 $\langle Z\rangle =\langle X\rangle$.
 \end{itemize}
 \end{prop}

Let $\A$ be an algebra that satisfies the exchange property.
It follows from (d) of Proposition~\ref{506} that $\A$
has a maximal independent set.  Any such set -- which must necessarily generate $\A$ -- is called a \emph{basis} for $\A$.
Moreover, all {bases} for $\A$ may be characterized as minimal generating sets, and they all have the same cardinality.
This common cardinality of the bases is called the \emph{dimension} of $\A$, written $\dim(\A)$.

\begin{defi}
\label{dia}
{\rm
An algebra $\A$ is called an \emph{independence algebra} if
\begin{itemize}
  \item [\rm(1)] $\A$ satisfies the exchange property {\rm(\ref{eq1})}, and
  \item [\rm(2)] for any basis $X$ of $\A$, if $\al:X\to A$, then there is an endomorphism $\bar{\al}$ of $\A$
such that $\bar{\al}|_X=\al$.
\end{itemize}
}
\end{defi}

Condition (2) of Definition~\ref{dia} states that an independence algebra $\A$ is a free object in the variety it generates
and any basis for $\A$ is a set of free generators.

We will need the following lemmas about  algebras in general.

\begin{lemma}
\label{lequ}
If $\A_1$ and $\A_2$ are equivalent  algebras, then $\End(\A_1)=\End(\A_2)$.
\end{lemma}
\begin{proof}
The result follows immediately
from the definitions of an endomorphism and the clone of an algebra.
\end{proof}

A $k$-ary operation on an algebra $\A$ is called a \emph{constant operation} if
there is $a\in A$ such that $f(a_1,\ldots,a_k)=a$ for all $a_1,\ldots,a_k\in A$.

\begin{lemma}
\label{luna}
Let $\A=\lan A;F\ran$ be an algebra whose clone contains a constant operation. Then
$\clf$ contains a unary constant operation.
\end{lemma}
\begin{proof}
Let $f$ be a constant $k$-ary operation in $\clf$. Then $f$ is defined by $f(x_1,\ldots,x_k)=a$,
where $a\in A$. Since $\clf$ contains all projections $p^k_i$ and it is closed under generalized
composition, $h(x)=f(p^1_1(x),\ldots,p^1_1(x))=a$ is in $\clf$.
\end{proof}

\section{Classification of $v^*$-algebras}
\label{scla}
The $v^*$-algebras without nullary operations
were classified by Urbanik in the 1960s \cite{Ur63,Ur65,urbanik}. In this section, we present Urbanik's classification.
Throughout this section, $\A=\lan A;F\ran$ will be an algebra without nullary operations.

\begin{defi}
\label{dmia}
{\rm
Suppose that $A$ is a monoid such that every non-unit element of $A$ is a left zero.
We say that $\A$ is a \emph{monoid independence algebra}, and write $\A=\mia$,
if for every $f\in F$, $f$ is a $k$-ary operation with $k\geq1$ such that for all $a,a_1,\ldots,a_k\in A$,
\begin{equation}
\label{dmiaeq0}
f(a_1a,\ldots,a_ka)=f(a_1,\ldots,a_k)a,
\end{equation}
and $F$ contains all unary operations that satisfy (\ref{dmiaeq0}). It is easy to see that every
unary operation $f$ that satisfies (\ref{dmiaeq0}) is defined by $f(x)=bx$, where $b\in A$.
}
\end{defi}

Let $A$ be a non-empty set on which two binary operations are defined: a multiplication $(a,b)\to ab$
and a subtraction $(a,b)\to a-b$. We say that $A$ is a \emph{quasifield} \cite{Gr63} if there is $0\in A$ such that
$a0=0a=0$ for every $a\in A$, $A\setminus\{0\}$ is a group with respect to the multiplication,
and for all $a,b,c\in A$, the following properties are satisfied:
\begin{itemize}
  \item [(i)] $a-0=a$,
  \item [(ii)] $a(b-c)=ab-ac$,
  \item [(iii)] $a-(a-c)=c$,
  \item [(iv)] $a-(b-c)=(a-b)-(a-b)(b-a)^{-1}c$\hskip 2mm if $a\ne b$.
\end{itemize}

\begin{defi}
\label{dqia}
{\rm
Suppose that $A$ is a quasifield.

We say that $\A$ is a \emph{quasifield independence algebra}, and write $\A=\qia$, if for every $f\in F$,
$f$ is a $k$-ary operation with $k\geq1$ such that for all $a,b,a_1,\ldots,a_k\in A$,
\begin{equation}
\label{dqiaeq1}
f(a-ba_1,\ldots,a-ba_k)=a-bf(a_1,\ldots,a_k),
\end{equation}
and $F$ contains all binary operations that satisfy (\ref{dqiaeq1}).
}
\end{defi}

\begin{defi}
\label{deia}
{\rm
Suppose that $A$ has four elements and $F=\{i,q\}$, where $i$ is a unary operation and $q$ is a ternary operation.

We say that $\A$ is the \emph{exceptional independence algebra}, and write $\eia=(A,i,q)$, if
$i$ is an involution without fixed points ($i(i(x))=x$ and $i(x)\ne x$ for every $x\in A$)
and $q$ is symmetrical ($q(x_1,x_2,x_3)=q(x_{\si(1)},x_{\si(2)},x_{\si(3)})$ for all
$x_1,x_2,x_3\in A$ and all permutations $\si$ of $\{1,2,3\}$) such that
for all $x,y\in A$, $q(x,y,i(x))=y$ and $q(x,y,x)=x$. (One can check that $q$ is uniquely determined
by these conditions.)
}
\end{defi}

\begin{defi}
\label{dgia}
{\rm
Suppose that $G$ is a group of permutations of $A$, and that $A_0$ is a subset of $A$ such that:
(i) all fixed points of any non-identity $g\in G$ are in $A_0$, and (ii) for every $g\in G$, $g(A_0)\subseteq A_0$.

We say that $\A$ is a \emph{group action independence algebra}, and write $\gia=(A,A_0,G)$, if
$\clf$ consists of all operations defined by
\[
f(x_1,\ldots,x_k)=g(x_j)\,\,\mbox{ or }\,\, f(x_1,\ldots,x_k)=a,
\]
where $k\geq1$, $1\leq j\leq k$, $g\in G$, and $a\in A_0$.
}
\end{defi}

\begin{defi}
\label{dlia}
{\rm
Suppose that $A$ is a linear space over a division ring $K$, and that $A_0$ is a linear subspace of $A$.

We say that $\A$ is a \emph{linear independence algebra}, and write $\lia=(A,A_0,K)$, if
$\clf$ consists of all operations defined by
\[
f(x_1,\ldots,x_k)=\sum_{i=1}^k\lam_ix_i+a,
\]
where $k\geq1$, each $\lam_i\in K$, and $a\in A_0$.
}
\end{defi}

\begin{defi}
\label{daia}
{\rm
Suppose that $A$ is a linear space over a division ring $K$, and that $A_0$ is a linear subspace of $A$.

We say that $\A$ is an \emph{affine independence algebra}, and write $\aia=(A,A_0,K)$, if
$\clf$ consists of all operations defined by
\[
f(x_1,\ldots,x_k)=\sum_{i=1}^k\lam_ix_i+a,
\]
where $k\geq1$, each $\lam_i\in K$, $\sum_{i=1}^k\lam_i=1$, and $a\in A_0$.
}
\end{defi}

 Algebras from Definitions~\ref{dmia}--\ref{daia} are independence algebras, which is already reflected
in their names. Moreover, they exhaust all possible independence algebras without nullary operations. This is due to Urbanik's classification
theorem \cite{Ur63,Ur65,urbanik}.

\begin{theorem}
\label{tcla}
Let $\A$ be an independence algebra without nullary operations of dimension at least $1$. Then $\A$ is one of the following:
\begin{itemize}
  \item [\rm(a)] a monoid independence algebra $\mia$;
  \item [\rm(b)] a quasifield independence algebra $\qia$;
  \item [\rm(c)] the exceptional independence algebra $\eia=(A,i,q)$;
  \item [\rm(d)] a group action independence algebra $\gia=(A,A_0,G)$;
  \item [\rm(e)] a linear independence algebra $\lia=(A,A_0,K)$;
  \item [\rm(f)] an affine independence algebra $\aia=(A,A_0,K)$.
\end{itemize}
Moreover, $\dim(\mia)=1$, $\dim(\qia)=2$, $\dim(\eia)=2$, $\dim(\gia)=n$,
where $n$ is the number of
$G$-transitive components of $A\setminus A_0$, $\dim(\lia)=n$,
where $n$ the linear dimension of the quotient space $A/A_0$, and $\dim(\aia)=n+1$,
where $n$ is the linear dimension of the quotient space $A/A_0$.
\end{theorem}

%\section{Constant independence algebras}
%\label{scia}
%Let $\cia$ be a constant independence algebra, that is, an algebra such that every element
%of $A$ is an algebraic constant. These algebras are the only zero-dimensional $v^*$-algebras \cite[page~242]{urbanik}.
%Since the only endomorphism of $\cia$ is the identity,
%a largest subsemilattice of $\End(\cia)$ has one element.

\section{Monoid independence algebras}
\label{scia}
In this section, $\mia$ will denote a finite-dimensional monoid independence algebra
(see Definition~\ref{dmia}). The dimension of $\mia$ is $1$, with each invertible element of the monoid $A$
forming a basis for $\mia$ \cite[page~242]{urbanik}.
We will determine the size of a largest subsemilattice of $\End(\mia)$.

\begin{lemma}
\label{la1}
Let $\al\in\End(\mia)$. Then $\al(a)=a\al(1)$ for every $a\in A$.
\end{lemma}
\begin{proof}
Let $a\in A$. Since $\al$ preserves the operation $f(x)=ax$,
we have $\al(a)=\al(a1)=\al(f(1))=f(\al(1))=a\al(1)$.
\end{proof}
\begin{lemma}
\label{lexx}
Let $\vep\in\End(\mia)$ be an idempotent such that $\vep(1)$ is invertible. Then $\vep(1)=1$.
\end{lemma}
\begin{proof}
By Lemma~\ref{la1}, $\vep(1)=\vep(\vep(1))=\vep(1)\vep(1)$, and so $\vep(1)=1$ since $\vep(1)$ is invertible.
\end{proof}
\begin{lemma}
\label{lhub}
Let $\vep_1,\vep_2\in\End(\mia)$ be commuting idempotents different from the identity.
Then $\vep_1=\vep_2$.
\end{lemma}
\begin{proof}
Since $\vep_1,\vep_2\ne\id_A$, $\vep_1(1)$ and $\vep_2(1)$ are not invertible by Lemmas~\ref{la1} and~\ref{lexx}.
Thus, $\vep_1(1)$ and $\vep_2(1)$ are left zeros in the monoid $A$, and so, by Lemma~\ref{la1},
$(\vep_1\vep_2)(1)=\vep_1(\vep_2(1))=\vep_2(1)\vep_1(1)=\vep_2(1)$. Similarly, $(\vep_2\vep_1)(1)=\vep_1(1)$,
and so $\vep_1(1)=\vep_2(1)$ since $\vep_1\vep_2=\vep_2\vep_1$.
Hence, for every $a\in A$, $\vep_1(a)=a\vep_1(1)=a\vep_2(1)=\vep_2(a)$, and so $\vep_1=\vep_2$.
\end{proof}

\begin{theorem}\label{tmia}
Let $\mia$ be a monoid independence algebra, and let $E$ be a largest subsemilattice
of $\End(\mia)$. Then $|E|=1$ if the monoid $A$ is a group, and $|E|=2$ if $A$ is not a group.
\end{theorem}
\begin{proof}
Suppose that $A$ is a group. Let $\vep\in\End(\mia)$ be an idempotent. Since $A$ is a group,
$\vep(1)$ is invertible, and so $\vep(1)=1$ by Lemma~\ref{lexx}. Thus, by Lemma~\ref{la1},
$\vep(a)=a\vep(1)=a1=a$, and so $\vep$ is the identity. It follows that $|E|=1$.

Suppose that $A$ is not a group. First, $|E|\leq2$ by Lemma~\ref{lhub}. Next,
fix a non-invertible element $c_0\in A$ and define $\vep:A\to A$ by $\vep(a)=ac_0$.
It is then straightforward to check that $\vep$ is an idempotent endomorphism of $\mia$.
Thus $L=\{\id_A,\vep\}$ is a subsemilattice of $\End(\mia)$, and so $|E|\geq|L|=2$. Hence $|E|=2$.
\end{proof}

\section{Quasifield independence algebras}
\label{smia}
In this section, $\qia$ will denote a finite-dimensional quasifield independence algebra
(see Definition~\ref{dqia}). The dimension of $\qia$ is $2$, with any two distinct elements of the quasifield $A$
forming a basis for $\qia$ \cite[page~243]{urbanik}.
We will determine the size of a largest subsemilattice of $\End(\qia)$.

It easily follows from the axioms for a quasifield that for all $a,b\in A$,
\begin{equation}
\label{eqf}
a=b\miff a-b=0.
\end{equation}
\begin{lemma}
\label{lcon}
Let $\al:A\to A$ be a constant transformation. Then $\al\in\End(\qia)$.
\end{lemma}
\begin{proof}
Suppose $\al(x)=c$ for every $x\in A$, where $c\in A$.
Let $f$ be an operation in $\qia$. Then, by the definition of $\qia$, $f$ is a $k$-ary operation ($k\geq1$) such that
$f(a-ba_1,\ldots,a-ba_k)=a-bf(a_1,\ldots,a_k)$ for all $a,b,a_1,\ldots,a_k\in A$. Then
\[
f(\al(x_1),\ldots,\al(x_k))=f(c,\ldots,c)=f(c-0c,\ldots,c-0c)=c-0f(c,\ldots,c)=c=\al(f(x_1,\ldots,x_k)).
\]
Hence $\al$ preserves $f$, and so $\al\in\End(\qia)$.
\end{proof}

\begin{theorem}\label{tqia}
Let $\qia$ be a quasifield independence algebra, and let $E$ be a largest subsemilattice
of $\End(\qia)$. Then $|E|=2$.
\end{theorem}
\begin{proof}
Let $\vep\in\End(\qia)$ be an idempotent other then the identity. Then, there is
$a\in A$ such that $a \ne \vep(a)$. We claim that $\vep$ is a constant transformation.

Let $b\in A$ be any element such that $b\ne a$. Our objective is to show that $\vep(b)=\vep(a)$.
By (\ref{eqf}), $a-b\ne0$ and $a-\vep(a)\ne0$, which implies
that there exists $c\in A \setminus \{0\}$ such that $a-b = (a-\vep(a))c$.
Consider the operation $f(x,y)= x-(x-y)c$ of $\qia$. Then
$f(a,\vep(a))=a-(a-\vep(a))c=a-(a-b)=b$, and so, since $\vep$ is an idempotent and $\vep$ preserves $f$,
\[
\vep(b)=\vep(f(a,\vep(a)))=f(\vep(a),\vep(\vep(a)))=f(\vep(a),\vep(a))=\vep(a)-(\vep(a)-\vep(a))c=\vep(a)-0c=\vep(a).
\]
We have proved that $\vep$ is a constant transformation. It follows that $|E|\leq2$.
Since $|A|\geq2$, there is a constant transformation $\vep$ of $A$ such that $\vep\ne\id_A$.
Then $L=\{\id_A,\vep\}$ is a subsemilattice of $\qia$, and so $|E|\geq|L|=2$. Hence $|E|=2$.
\end{proof}

\section{The exceptional independence algebra}
Let $\eia=(A,i,q)$ be the exceptional independence algebra
(see Definition~\ref{deia}). The dimension of $\eia$ is $2$, with any two distinct elements $x,y\in A$
such that $y\ne i(x)$
forming a basis for $\eia$ \cite[page~244]{urbanik}.

\begin{theorem}
\label{teia}
Let $\eia$ be the exceptional independence algebra, and let $E$ be a largest subsemilattice
of $\End(\eia)$. Then $|E|=2$.
\end{theorem}
\begin{proof}
Let $A=\{a,b,c,d\}$. We may assume that the involution $i$, written in cycle notation,
is $i=(a\,b)(c\,d)$. Let $\vep\in\End(\eia)$ be an idempotent. Then $\vep$ preserves $i$,
and so $i(\vep(a))=\vep(i(a))=\vep(b)$. Thus $(\vep(a)\,\vep(b))$ is a cycle in $i$. Similarly,
$(\vep(c)\,\vep(d))$ is also a cycle in $i$. Moreover, since $\vep$ is an idempotent,
if $(x\,y)$ is a cycle in $i$ such that $x,y\in\ima(\vep)$, then $\vep(x)=x$ and $\vep(y)=y$. It follows
that $\vep$ must be one of the following transformations of $A$:
\begin{equation}
\label{teiaeq1}
\begin{pmatrix}a&b&c&d\\a&b&c&d\end{pmatrix},\,\,
\begin{pmatrix}a&b&c&d\\a&b&a&b\end{pmatrix},\,\,
\begin{pmatrix}a&b&c&d\\a&b&b&a\end{pmatrix},\,\,
\begin{pmatrix}a&b&c&d\\c&d&c&d\end{pmatrix},\,\,\mbox{or}\,\,
\begin{pmatrix}a&b&c&d\\d&c&c&d\end{pmatrix}.
\end{equation}
It is easy to check that if $\vep_1$ and $\vep_2$ are two distinct commuting idempotents that occur in (\ref{teiaeq1}),
then either $\vep_1$ or $\vep_2$ must be the identity. It follows that $|E|\leq2$.
Since $\vep=\begin{pmatrix}a&b&c&d\\a&b&a&b\end{pmatrix}$ preserves both $i$ and $q$,
$L=\{\id_A,\vep\}$ is a subsemilattice of $\End(\eia)$, and so $|E|\geq|L|=2$. Hence $|E|=2$.
\end{proof}

\section{Group action independence algebras}
In this section, $\gia=(A,A_0,G)$ will denote a finite-dimensional group action independence algebra
(see Definition~\ref{dgia}). We fix a cross-section $X$ of the set of
$G$-transitive components of $A\setminus A_0$. By \cite[page~244]{urbanik}, $X$ is a basis for $\gia$, so
$\gia$ has dimension $|X|$.
We will determine the size of a largest subsemilattice of $\End(\gia)$.

Denote by $\A_1=\lan A;F_1\ran$ the  algebra such that $F_1$ consists of the unary operations $g$,
where $g\in G$, and all constant transformations $f(x)=a$, where $a\in A_0$. Then $\A$ and $\A_1$ are equivalent.

\begin{lemma}\label{lsua}
The subalgebras of $\A_1$ are the sets $G(Y)\cup A_0$, where $Y\subseteq X$.
\end{lemma}
\begin{proof}
It is clear that $G(Y)\cup A_0$ is a subalgebra of $\A_1$ for every $Y\subseteq X$.
Let $B$ be a subalgebra of $\A_1$, and let $Y=B\cap X$. We claim that $B=G(Y)\cup A_0$.
First, $G(Y)\cup A_0\subseteq B$ since $Y\subseteq B$ and $A_0$ is a subset of any subalgebra of $\A_1$.
Let $a\in B$. If $a\in A_0$, then $a\in G(Y)\cup A_0$. Suppose $a\in A\setminus A_0$.
Then $a$ must be in some $G$-transitive component of $A\setminus A_0$, and so $a=g(x)$ for some $x\in X$.
Thus $x=g^{-1}(a)\in B\cap X=Y$, and so $a\in G(Y)$. Hence $B\subseteq G(Y)\cup A_0$.
\end{proof}

\begin{theorem}\label{tgia}
Let $\gia=(A,A_0,G)$ be a group action independence algebra of finite dimension~$n$, and let $E$ be a largest subsemilattice
of $\End(\gia)$. Then $|E|=2^{n-1}$ if $A_0=\emptyset$, and $|E|=2^n$ if $A_0\ne\emptyset$.
\end{theorem}
\begin{proof}
Recall that $n=|X|$, where $X$ is our fixed cross-section of the $G$-transitive components of $A\setminus A_0$.
Since $\End(\gia)=\End(\A_1)$ (see Lemma~\ref{lequ}),
we may assume that $E$ is a largest subsemilattice of $\A_1$. Suppose that $A_0=\emptyset$.
Then, by Lemma~\ref{lsua}, the subalgebras of $\A_1$ are the sets $G(Y)$, where $Y$ is a nonempty subset of $X$.
Since $X$ is finite, it follows that $\A_1$ has finitely many subalgebras.

We claim that $E$ is finite.
Let $\vep\in E$. Since $\vep$ is an endomorphism of $\A_1$, $\ima(\vep)$ is a subalgebra of $\A_1$.
We will now show that different elements of $E$ have different images.
Indeed, suppose that $\vep_1,\vep_2\in E$ with $\ima(\vep_1)=\ima(\vep_2)$.
Then for every $a\in\A_1$, $(\vep_1\vep_2)(a)=\vep_1(\vep_2(a))=\vep_2(a)$ since $\vep_2(a)\in\ima(\vep_1)$
and an idempotent fixes every element of its image. Thus $\vep_2=\vep_1\vep_2$. Similarly, $\vep_1=\vep_2\vep_1$,
and so $\vep_1=\vep_2$ since $\vep_1\vep_2=\vep_2\vep_1$.
It follows that the number of elements of $E$ cannot be greater than the number of subalgebras of $\A_1$.
Since the latter is finite, $E$ is also finite, say $E=\{\vep_1,\ldots,\vep_k\}$.

The minimal subalgebras of $\A_1$ are the subsets
$G(\{x\})$, where $x\in X$. Since the subalgebra $\ima(\vep_1\cdots\vep_k)$ is included in $\ima(\vep_i)$
for every $i\in\{1,\ldots,k\}$, there is some $x_0\in X$ such that
$G(x_0)\subseteq\ima(\vep_i)$ for every $i\in\{1,\ldots,k\}$. Suppose that $\vep_i,\vep_j\in E$
with $\ima(\vep_i)\setminus G(x_0)=\ima(\vep_j)\setminus G(x_0)$. Then $\ima(\vep_i)=\ima(\vep_j)$,
and so, by the previous paragraph, $\vep_i=\vep_j$.
Therefore,
the mapping $\vep_i\to Y_i\setminus\{x_0\}$ (where $\ima(\vep_i)=G(Y_i)$) from $E$ to $\mathcal{P}(X\setminus\{x_0\})$ is injective,
and so $|E|\leq2^{|X|-1}=2^{n-1}$.

Now, fix $x_0\in X$ and let $\Gamma=\{Y\subseteq X: x_0\in Y\}$.
Then, for each $Y\in\Gamma$, define $\ey:A\mapsto A$  as follows: $\ey(a)=a$ if $a=g(x)$ and $x\in Y$,
and $\ey(a)=g(x_0)$ if $a=g(x)$ and $x\notin Y$.
Suppose $a\in A$. Since $A_0=\emptyset$,  $a$ must lie in some $G$-transitive component of $A\setminus A_0$. Thus
$a=g(x)$ for some $g\in G$ and $x\in X$, and such an $x$ is unique since $X$ is a cross-section of
the $G$-transitive components of $A\setminus A_0$. Suppose $g(x)=h(x)$, where $h\in G$.
Then $(h^{-1}g)(x)=x$, and so $h^{-1}g=1$ since all fixed points of the non-identity elements of $G$ are in $A_0$.
Hence $g=h$, and so $g$ is unique too.
We have proved that $\ey$ is well defined. It is easy to check that
$\ey$ is an idempotent in $\End(\A_1)$. Let $L=\{\ey:Y\in \Gamma\}.$ Then $L$ is a subsemilattice of $\A_1$
since for all $Y,Z\in\Gamma$, $\ey\ez=\eyz$. Moreover, $|L|=|\Gamma|=2^{|X|-1}=2^{n-1}$.
Therefore, $|E|\geq|L|=2^{n-1}$, and so $|E|=2^{n-1}$.

Suppose that $A_0\ne\emptyset$.
By Lemma~\ref{lsua}, the subalgebras of $\A_1$ are the sets $G(Y)\cup A_0$, where $Y\subseteq X$.
As in the first part of the proof, the number of elements of $E$ cannot be greater than the number of subalgebras of $\A_1$,
hence $|E|\leq2^{|X|}=2^n$.
Now, fix $a_0\in A_0$ and,
for each $Y\subseteq X$, define $\ey:A\mapsto A$  as follows: $\ey(a)=a$ if $a=g(x)$ for some $g\in G$ and $x\in Y$,
$\ey(a)=a_0$ if $a=g(x)$ and for some $g\in G$ and $x\notin Y$, and $\ey(a)=a$ if $a\in A_0$.
As in the first part of the proof,
$\ey$ is a well-defined
idempotent in $\End(\A_1)$. Let $L=\{\ey:Y\subseteq X\}.$ Then $L$ is a subsemilattice of $\A_1$
since for all $Y,Z\subseteq X$, $\ey\ez=\eyz$. Moreover, $|L|=2^{|X|}=2^n$.
Therefore, $|E|\geq|L|=2^n$, and so $|L|=2^n$.
\end{proof}

\section{Linear independence algebras}
\label{slia}
In this section, $\lia=(A,A_0,K)$ will denote a finite-dimensional linear independence algebra
(see Definition~\ref{dlia}). If $S_0$ is a basis of the linear subspace
$A_0$ and $S$ is a basis of the linear space $A$ that is an extension of $S_0$, then $T=S\setminus S_0$
is a basis of the independence algebra $\lia$ \cite[page~236]{urbanik}, so $\lia$ has dimension $|T|$. In other
words, the dimension of $\lia$ is equal to the linear dimension of the quotient space $A/A_0$.
The monoid $\End(\lia)$ consists of all mappings from $A$ to $A$
that preserve every operation $f(x_1,\ldots,x_k)=\sum_{i=1}^k\lam_ix_i+a$,
where $k\geq1$, each $\lam_i\in K$, and $a\in A_0$. It easily follows that
$\End(\lia)$ consists of all linear transformations of $A$ that fix every element of $A_0$.
We will determine the size of a largest subsemilattice of $\End(\lia)$.

First, we need the following result from matrix algebra \cite[pages 51--53]{HJ85}.
\begin{lemma}\label{ldia}
Let $\mathcal{E}$ be a set of $m\times m$ pairwise commuting diagonalizable matrices over a division ring $K$.
Then there exists an invertible matrix $P$ such that $P^{-1}MP$ is diagonal for every $M\in\mathcal{E}$.
\end{lemma}

\begin{lemma}\label{lub}
Suppose that $\lia$ has dimension $n$. Let $E$ be a set of pairwise commuting idempotents in $\End(\lia)$.
Then $|E|\leq2^n$.
\end{lemma}
\begin{proof}
As idempotent endomorphisms, the elements of $E$ are pairwise commuting projections that contain $A_0$ in
their image. These elements are in a natural 1-1 correspondence with the projections of the linear quotient space $A/A_0$, and
moreover, the correspondence clearly preserve the property of being pairwise commutative. Hence we may assume that
$A_0=\{0\}$ and that $n$ is also the linear dimension of $\lia$.

With respect to a linear basis for $\lia$, projections are represented by diagonalizable matrices. By Lemma \ref{ldia} we can find a basis $B$ that
diagonalizes all matrices representing the elements of $E$ simultaneously.  With respect to $B$, the elements of $E$ are represented by diagonal
$n \times n$-matrices with values in $\{0,1\}$ (as the elements of $E$ are projections). The result follows as there are at most $2^n$ such matrices.
\end{proof}

\begin{theorem}
\label{tlia}
Let $\lia=(A,A_0,K)$ be a linear independence algebra of finite dimension $n$, and let $E$ be a largest subsemilattice
of $\End(\lia)$. Then $|E|=2^n$.
\end{theorem}
\begin{proof}
By Lemma~\ref{lub}, it suffices to construct a semilattice of endomorphisms of $\lia$ of size $2^n$.
Let $T$ be a basis for $\lia$ (as an independence algebra). For each subset $Y\subseteq T$,
define $\ey:A\to A$ by: $\ey(x)=x$ if $x\in Y$, and $\ey=0$ if $x\in T\setminus Y$. It is straightforward
to check that $L=\{\ey:Y\subseteq T\}$ is a subsemilattice of $\End(\lia)$ with $|L|=2^{|T|}=2^n$.
\end{proof}

\begin{cor}
Let $V$ be a vector space of dimension $n$. Then the largest semilattices of linear transformations of $V$ have size $2^n$.
\end{cor}

\section{Affine independence algebras}
\label{saia}
In this section, $\aia=(A,A_0,K)$ will denote a finite-dimensional affine independence algebra
(see Definition~\ref{daia}). The dimension of $\aia$ is one more than the linear dimension of the quotient space
$A/A_0$ \cite[page~236]{urbanik}. Hence the dimension of $\aia$ is one more than the dimension of the corresponding
linear independence algebra $\lia=(A,A_0,K)$.
The monoid $\End(\aia)$ consists of all mappings from $A$ to $A$
that preserve every operation $f(x_1,\ldots,x_k)=\sum_{i=1}^k\lam_ix_i+a$,
where $k\geq1$, each $\lam_i\in K$, $\sum_{i=1}^k\lam_i=1$, and $a\in A_0$.

\begin{theorem}\label{taia}
Let $\aia=(A,A_0,K)$ be an affine independence algebra of finite dimension $n$, and let $E$ be a largest subsemilattice
of $\End(\aia)$. Then $|E|=2^{n-1}$.
\end{theorem}
\begin{proof}
The monoid
$\End(\aia)$ consists of all affine transformations of $A$ whose restrictions to $A_0$ are translations.
Hence, every idempotent in $\End(\aia)$ is a projection onto an affine subspace that contains a translate of $A_0$. As in the linear case,
we may assume that $A_0 =\{0\}$, that the idempotent elements of $\End(\aia)$ are the projections onto affine subspaces, and
that $n$ is one more than the linear dimension of $\lia$.

Every affine projection $\vep$ can be written uniquely in the form $\vep(x)=\al_\vep(x)+c_\vep$, where $\al_\vep$ is a linear projection
and $c_\vep\in A$.
Let $\vep_1,\vep_2 \in E$ with $\vep_1(x)=\al_1(x)+c_1$ and $\vep_2(x)=\al_2(x)+c_2$.

We claim that $\al_1$ and $\al_2$ commute.
Since $\vep_1$ and $\vep_2$ commute, we obtain, for every $x\in A$,
\begin{align}
\al_1(\al_2(x))+\al_1(c_2)+c_1&=\al_1(\al_2(x)+c_2)+c_1=\vep_1(\al_2(x)+c_2)=\vep_1(\vep_2(x))=\vep_2(\vep_1(x))\notag\\
&=\vep_2(\al_1(x)+c_1)=\al_2(\al_1(x)+c_1)+c_2= \al_2(\al_1(x))+\al_2(c_1)+c_2.\notag
\end{align}
Setting $x=0$, we obtain $\al_1(c_2)+c_1=\al_2(c_1)+ c_2$, which implies that
$\al_1(\al_2(x))=\al_2(\al_1(x))$ for every $x\in A$. This proves the claim.

Let $L=\{\al_\vep:\vep\in E\}$. By the claim, $L$ is a subsemilattice of $\End(\lia)$.
Define $\varphi:E\to L$ by $\varphi(\vep)=\al_\vep$.
Let $\vep_1,\vep_2 \in E$ such that $\vep_1(x)=\al(x)+c_1$ and $\vep_2(x)=\al(x)+c_2$ (that is, $\varphi(\vep_1)=\varphi(\vep_2)$).
If $c_1 \ne c_2$, then $\vep_1(x)=\al(x)+c_1$ and $\vep_2(x)=\al(x)+c_2$ have disjoint images, and hence they do not commute.
It follows that $\varphi$ is injective. Therefore,
$|E|\leq|L|$. By Theorem~\ref{tlia}, $|L|\leq 2^m$, where $m$ is the dimension of the linear independence
algebra $\lia$. Since $n=m+1$, it follows that $|E|\leq2^{n-1}$.

Conversely, by Theorem~\ref{tlia}, there is a subsemilattice $L_1$ of $\End(\lia)$ with $|L_1|=2^m=2^{n-1}$.
It is clear by the definition of $\lia$ and $\aia$ that $\End(\lia)\subseteq\End(\aia)$.
Hence $L_1$ is also a subsemilattice of $\End(\aia)$, and so
$|E|\geq|L_1|=2^{n-1}$. The result follows.
\end{proof}

\section{The main theorem}
\label{smth}
Our main result will follow from Theorem~\ref{tcla} and the following proposition.

\begin{prop}
\label{plink}
Let $\A=\lan A;F\ran$ be an independence algebra with $\dim(\A)\geq1$. Let $\A_1=\lan A;F_1\ran$, where
$F_1=\clf\setminus\con$. Then $\A_1$ is an independence
algebra with $\dim(\A_1)=\dim(\A)$ and $\End(\A_1)=\End(\A)$.
\end{prop}
\begin{proof}
We first prove that $\End(\A_1)=\End(\A)$.
Since $F_1\subseteq\clf$, we have $\End(\A)\subseteq\End(\A_1)$. Conversely, let $\al\in\End(\A_1)$ and let $f\in\clf$.
If $f$ is not a nullary operation, then $f\in F_1$, and so $\al$ preserves $f$.
Suppose $f=a\in A$ is a nullary operation.
By the definition of $\clf$, there exists a unary constant operation $g\in\clf$ with image $\{a\}$.
Then $g\in F_1$, and so $\al$ preserves $g$. Hence $\al(a)=\al(g(a))=g(\al(a))=a$, and so $\al$ preserves $f$.
It follows that $\End(\A_1)=\End(\A)$.

For $X\subseteq A$, we denote by $\lan X\ran_{\clf}$ and $\lan X\ran_{F_1}$ the closure of $X$ in $\A$
and $\A_1$, respectively. Let $\emptyset\ne X\subseteq A$. We claim that $\lan X\ran_{F_1}=\lan X\ran_{\clf}$.
Since $F_1\subseteq\clf$, we have $\lan X\ran_{F_1}\subseteq\lan X\ran_{\clf}$. Let $a\in\lan X\ran_{\clf}$.
Then $a=f(c_1,\ldots,c_k)$ ($k\geq0$) for some $f\in\clf$ and $c_1,\ldots,c_k\in X$. If $k\geq1$, then
$f\in F_1$, and so $a\in\lan X\ran_{F_1}$.
Suppose $k=0$. Then $a$ is a nullary operation in $\clf$. As in the previous paragraph, we then have a unary
operation $g\in F_1$ such that $g(x)=a$ for every $x\in A$. Since $X\ne\emptyset$, there is some $c\in X$.
Thus $a=g(c)\in\lan X\ran_{F_1}$. It follows that $\lan X\ran_{F_1}=\lan X\ran_{\clf}$.

By the claim and the fact that $\A$ satisfies the exchange property (\ref{eq1}), $\A_1$ also satisfies
the exchange property.
Let $X$ be a basis for $\A_1$. Since $\A_1$ does not have any nullary operations,
we must have $X\ne\emptyset$. We claim that $X$ is an independent set for $\A$ if $|X| \ge 2$.  For any
$x \in X$, $|X \setminus \{x\}| \ge 1$
and  the claim follows from $\lan X \setminus \{x\} \ran_{F_1}=\lan X\setminus \{x\} \ran_{\clf}$, and the fact that
$X$ is an independent set for $\A_1$.
As $\lan X\ran_{F_1}=A=\lan X\ran_{\clf}$, it is also a basis for
 for $\A$.
This implies $\dim(\A_1)=\dim(\A)$. Moreover, since $\End(\A_1)=\End(\A)$, every $\al:X\to A$ can be extended
to an endomorphism of $\End(\A_1)$. Hence $\A_1$ is an independence algebra.

Now let $|X|=1$. Then $A=\lan X \ran_{F_1} \subseteq \lan X \ran_{\clf}$, and as we assumed that
$\dim(\A) \ge 1$, we have that $\dim(\A)=1 = \dim(\A_1)$. If  every basis $X$ is also in independent set of $\A$, then  as above
we can conclude that $\A_1$ is an independence algebra. So suppose this is not the case. Then for some $X=\{x\}$,
$x \in \lan \emptyset \ran_{\clf}$,
which means that $x$ is the image of a nullary operation from $\clf$. Then $x$ would be the image of a constant unary
operation $u_x$ from $F_1$. As $\lan X \ran=A$ every element of $A$ would also be the image of a constant unary function
$t(u_x(x), \dots, u_x(x))$ from $F_1$, for some $t$.
But then every element of $A$ would also have the form $t(x,\dots,x)$, and hence $\lan \emptyset \ran_{\clf}=A$, contradicting
that $\dim(\A) \ge 1$.

\end{proof}

\begin{lemma}
\label{lea}
Let $\eia=(A,i,q)$ be the exceptional algebra. Then $\{i,q\}_{\mathrm{cl}}$ does not contain any constant operations.
\end{lemma}
\begin{proof}
It is straightforward to check that $i$ is an endomorphism of $\eia$. As $i$ has no fixed points, $\{i,q\}_{\mathrm{cl}}$
cannot contain any constant operations.
\end{proof}

Following \cite{urbanik}, we will say that $a\in A$ is an \emph{algebraic constant}
if there is $f\in\clf$ such that $f$ is a constant operation with image $\{a\}$.

\begin{theorem}
\label{tmain}
Let $\A=\lan A;F\ran$ be an independence algebra of finite dimension $n$, with $|A|\geq2$. Then the largest subsemilattices
of $\End(\A)$ have either $2^{n-1}$ or $2^n$ elements, with the latter happening exactly
when  the clone of $\A$ contains constant operations.
\end{theorem}
\begin{proof}
Let $E$ be a largest subsemilattice of $\End(\A)$. Suppose $\dim(\A)=0$. This can only happen when
$\A$ has nullary operations that generate $\A$. But then the only endomorphism of $\A$ is the identity,
and so $|E|=1=2^n$.

Suppose that $\dim(\A)\geq1$. By Proposition~\ref{plink}, we may assume that $\A$ does not have any nullary operations.
Then $\A$ is equivalent
to one of the independence algebras listed in Theorem~\ref{tcla}.

Suppose $\A=\mia$ is a monoid independence algebra.
Then $n=1$. Suppose $A$ is a group. Recall that any unary
operation $f$ of $\mia$ is defined by $f(x)=bx$, where $b\in A$ (see Definition~\ref{dmia}). Thus $f$ is not constant since $|A|\geq2$.
The set of operations that satisfy (\ref{dmiaeq0}) is closed under generalized composition \cite[Section~4.2]{urbanik}.
It follows that every unary operation $f$ in $\clf$ is defined by $f(x)=bx$, where $b\in A$.
Thus, by Lemma~\ref{luna}, $\clf$ does not contain any constant operations. By Theorem~\ref{tmia}, $|E|=1=2^{n-1}$.
Suppose $A$ is not a group and select a non-unit $a\in A$. Then $a$ is a left zero in the monoid $A$,
and so the operation $f(x)=ax=a$ is constant. Since $f\in F$, $\clf$ contains constant operations.
By Theorem~\ref{tmia}, $|E|=2=2^n$.

Suppose $\A=\qia$ is a quasifield independence algebra. Then $n=2$.
The only possible unary operation of $\qia$ is the identity \cite[Section~4.3]{urbanik}.
Thus $f$ is not constant since $|A|\geq2$.
The set of operations that satisfy (\ref{dqiaeq1}) is closed under generalized composition \cite[Section~4.3]{urbanik}.
It follows that the only unary operation in $\clf$ is the identity.
Thus, by Lemma~\ref{luna}, $\clf$ does not contain any constant operations. By Theorem~\ref{tqia}, $|E|=2=2^{n-1}$.

Suppose $\A=\eia=(A,i,q)$ is the exceptional independence algebra. Then $n=2$.
By Lemma~\ref{lea}, $\{q,i\}_{\mathrm{cl}}$ does not contain any constant operations.
By Theorem~\ref{teia}, $|E|=2=2^{n-1}$.

Suppose $\A=\gia=(A,A_0,G)$ is a group action independence algebra. By Definition~\ref{dqia},
the clone of $\gia$ contains constant operations if and only if $A_0\ne\emptyset$.
Thus, by Theorem~\ref{tgia}, $|E|=2^{n-1}$ if the clone of $\gia$ does not contain any constant operations,
and $|E|=2^n$ if the clone of $\gia$ contains constant operations.

Suppose $\A=\lia=(A,A_0,K)$ is a linear independence algebra. By Definition~\ref{dlia},
the clone of $\lia$ contains constant operations. By Theorem~\ref{tlia}, $|E|=2^n$.

Suppose $\A=\aia=(A,A_0,K)$ is an affine independence algebra. By Definition~\ref{daia},
the clone of $\lia$ does not contain any constant operations. By Theorem~\ref{taia}, $|E|=2^{n-1}$.

Hence, the result follows from Theorem~\ref{tcla} and the fact that equivalent  algebras
have the same monoids of endomorphisms.
\end{proof}

Let $\A=\lan\{a\};F\ran$ and let $n=\dim(\A)$. Then $\End(\A)=\{\id_A\}$, so $|E|=1$.
The clone of $\A$ contains constant operations since all projections are constant.
If $a$ is a nullary operation of $\A$, then $n=0$, and so $|E|=1=2^n$ (that is, Theorem~\ref{tmain} holds).
If $a$ is not a nullary operation of $\A$, then $n=1$, and so $|E|=1\ne2^n$ (that is, Theorem~\ref{tmain} fails).

\section{Problems}
\label{spro}

In this section,
we present some problems that might attract attention of experts in linear algebra, ring theory, extremal combinatorics, group theory, semigroup theory,
universal algebraic geometry, and universal algebra.

The overwhelming majority of theorems proved about independence algebras do not use any classification theorems.
Therefore, the next problem is natural.

\begin{prob}
{\rm
Is it possible to prove Theorem \ref{tmain} without using the classification theorem for $v^*$-algebras?
}
\end{prob}

Fountain and Gould defined the class of \emph{weak exchange algebras}, which contains independence algebras, weak
independence algebras, and basis algebras, among others \cite{FoGo03,FoGo04}.

\begin{prob}
{\rm
Prove an analogue of Theorem \ref{tmain} for the class of weak exchange algebras
and its subclasses.
}
\end{prob}

As we have already indicated,
the study of graphs induced by  algebras has attracted a great deal of attention in recent years.
The commuting graph of a semigroup $S$ is the graph whose vertices are the non-central elements of
$S$ and two of them form an edge if  they commute as elements of the semigroup.
The idempotent commuting graph of a semigroup $S$ is the commuting graph of $S$
restricted to the idempotent elements.
Our main theorem provides the clique number
of the idempotent commuting graphs of the monoids $\End(\A)$, when $\A$ is an independence algebra.

\begin{prob}
\label{pgra}
{\rm
Let $\A$ be an independence algebra.
Find the diameter of the idempotent commuting graph of $\End(\A)$.
Similarly, find the diameter of the commuting graph of $\End(\A)$.
(Some related results are contained in \cite{ArKiKo}.)
}
\end{prob}

The following problem may be difficult to solve in all generality, but some progress might be achieved
using the classification theorem for $v^*$-algebras.

\begin{prob}
\label{pcli}
{\rm
Find the clique number of the commuting graph of $\End(\A)$, where $\A$ is an independence algebra.
}
\end{prob}

We point out that the answer to this question is not known even in the case of $T(X)$,
the full transformation monoid on a finite set $X$.
The problem has been solved for the finite symmetric inverse semigroup \cite{ArBeKo}, but the complexity of the
arguments in that paper suggests that Problem~\ref{pcli} might be very hard for some classes of independence algebras.

The first author and Wehrung  \cite{ArWe} introduced a large number of classes of algebras that generalize
independence algebras.

\begin{prob}
{\rm
Prove a result similar to  Theorem~\ref{tmain} and
solve analogs of Problems~\ref{pgra} and~\ref{pcli}
for $MC$-algebras, $\mathit{MS}$-algebras, $\mathit{SC}$-algebras, and $\mathit{SC}$-ranked algebras  \cite[Chapter 8]{ArWe}.

A first step in solving these problems would be to find the size of the largest semilattice contained in the endomorphism monoid of an
$\mathit{SC}$-ranked free $M$-act \cite[Chapter 9]{ArWe},
and for an $\mathit{SC}$-ranked free module over an $\aleph_1$-Noetherian ring \cite[Chapter 10]{ArWe}.
}
\end{prob}

Finally, we suggest a family of problems that concern groups of automorphisms.

\begin{prob}
\label{plast}
{\rm
Let $\A$ be an algebra belonging to one of the  classes of algebras referred to above
(independence algebras, $v^*$-algebras, weak exchange algebras,
$\mathit{MC}$-algebras, $\mathit{MS}$-algebras,
$\mathit{SC}$-algebras, $\mathit{SC}$-ranked algebras, $\mathit{SC}$-ranked free $M$-act,
$\mathit{SC}$-ranked free modules over an $\aleph_1$-Noetherian ring, any of the classes in Urbanik's classification, etc.). Let $\End(\A)$ be  the monoid of endomorphisms of $\A$. Describe the automorphisms of $\End(\A)$.
}
\end{prob}

We observe that Problem~\ref{plast} is linked to one of the main questions
in \emph{universal algebraic geometry},
and it is still not solved for some classes of independence algebras, let alone for more general classes.
On the other hand,  many instances of this problem have been solved for other classes of algebras \cite{plotkin,plotkin2,Zi}.
(See also \cite{ArBuMiNeu,dobson,fernandes,ArKo7,ArKo7b,ArKo9a,ArKo9,Belov,Bel2,Be,MaPl}.)

\section*{Acknowledgements}
We are grateful to the referee for suggesting a simplification of the arguments in Sections~\ref{slia} and~\ref{saia},
and for other useful comments.

The  first author was partially supported by FCT through PEst-OE/MAT/UI0143/2011.
The second author  has received funding from the
European Union Seventh Framework Programme (FP7/2007-2013) under
grant agreement no.\ PCOFUND-GA-2009-246542 and from the Foundation for
Science and Technology of Portugal.

\end{document}